\documentclass[pdflatex,sn-mathphys-num]{sn-jnl}

\usepackage{graphicx}
\usepackage{multirow}
\usepackage{amsmath,amssymb,amsfonts}
\usepackage{amsthm}
\usepackage[title]{appendix}
\usepackage{xcolor}
\usepackage{textcomp}
\usepackage{manyfoot}
\usepackage{booktabs}
\usepackage{algorithm}
\usepackage{algorithmicx}
\usepackage{algpseudocode}
\usepackage{listings}
\usepackage{silence}
\usepackage[all]{xy}

\numberwithin{equation}{section}

\theoremstyle{thmstyleone}
\newtheorem{theorem}{Theorem}
\newtheorem{proposition}[theorem]{Proposition}

\newtheorem{lemma}{Lemma} 

\theoremstyle{thmstyletwo}
\newtheorem{example}{Example}
\newtheorem{remark}{Remark}

\theoremstyle{thmstylethree}
\newtheorem{definition}{Definition}

\vfuzz=\maxdimen
\hfuzz=\maxdimen

\newcommand{\Aut}{\operatorname{Aut}}
\newcommand{\Ker}{\operatorname{Ker}}
\newcommand{\Ext}{\operatorname{Ext}}
\newcommand{\Bim}{\operatorname{Bim}}
\newcommand{\Barc}{\operatorname{Bar}}
\newcommand{\Ob}{\operatorname{Ob}}
\newcommand{\Ima}{\operatorname{Im}}
\newcommand{\Hom}{\operatorname{Hom}}
\newcommand{\Mod}{\operatorname{Mod}}
\newcommand{\Hrm}{\operatorname{H}}
\newcommand{\HH}{\operatorname{HH}}

\newcommand{\Z}{\operatorname{\mathbb{Z}}}
\newcommand{\N}{\operatorname{\mathbb{N}}}

\begin{document}

\title[Equivariant Hochschild cohomology of  group algebras and relative $\Ext$]{Equivariant Hochschild cohomology of  group algebras and relative $\Ext$}

\author[1]{\fnm{Andrada} \sur{Pojar}}\email{Andrada.Pojar@math.utcluj.ro}

\author[2,3]{\fnm{Constantin-Cosmin} \sur{Todea}}\email{Constantin.Todea@math.utcluj.ro}

\affil[1]{\orgdiv{Department of Mathematics}, \orgname{Technical University of Cluj-Napoca}, \orgaddress{\street{25 G. Baritiu St}, \city{Cluj-Napoca}, \postcode{400027}, \country{Romania}}}

\affil[2]{\orgdiv{Department of Mathematics}, \orgname{Technical University of Cluj-Napoca}, \orgaddress{\street{25 G. Baritiu St}, \city{Cluj-Napoca}, \postcode{400027}, \country{Romania}}}

\affil[3]{\orgdiv{Department of Mathematics}, \orgname{Babes-Bolyai University}, \orgaddress{\street{1 Mihail Kogalniceanu St}, \city{Cluj-Napoca}, \postcode{400084}, \country{Romania}}}

\abstract{For a finite group  $\Gamma$, acting on a finite group $G,$ we find necessary conditions for which the first $\Gamma_0$-equivariant Hochschild cohomology of the group algebra $kG$ is non-trivial, where  $k$ is a field of characteristic $p$ dividing the order of $G$ and $\Gamma_0$ is the stabilizer subgroup in $\Gamma$ of some element in $G.$ For any field $k$ we show that the $\Gamma$-equivariant Hochschild cohomology of $\Gamma$-algebras with coefficients in a $\Gamma$-equivariant bimodule \citep{Je} is isomorphic with some $k\Gamma$-relative $\Ext,$ in the context of relative homological algebra.}



\maketitle

\section{Introduction}\label{sec1}

When there is a group action on an  algebra, this algebra is not studied in isolation but together with symmetries and then equivariant Hochschild cohomology (homology) shows up. There is a recent growing interest in studying equivariant Hochschild (co)homology \cite{KoPi}, \cite{In}, \cite{Mu}.

Let $\Gamma$ be a  finite group and $G$ be a  finite $\Gamma$-group, which is a group $(G,\cdot)$ endowed with a $\Gamma$-action by group automorphisms. This means that there is a group homomorphism $\Gamma\rightarrow \Aut(G,\cdot)$ defining  the $\Gamma$-action, that is denoted by $$\Gamma \times G \rightarrow G, (\sigma,g)\mapsto {}^{\sigma}g,$$ satisfying some well-known axioms. Sometimes, like in \citep{Ceg}, these groups are called  "groups with operators". Let $k$ be a field and $n\in \Z, n \geq 0.$ By a $\Gamma$-equivariant $kG$-module $N$ we understand a $kG$-module that is also a $k\Gamma$-module and the action of $G$ and $\Gamma$ are related to each other by the
equality $${}^{\sigma}(g m)={}^{\sigma}g({}^{\sigma} m),$$
for any $\sigma\in\Gamma, g\in G, m\in N.$ A structure of $\Gamma$-equivariant $kG$-module on $N$ is the same as a structure of left $k(G \rtimes \Gamma)$-module on $N,$ where $G \rtimes \Gamma$ is the semi-direct product of $G$ by $\Gamma,$ see \citep{Ceg}.

In Subsection \ref{subsec21} we revisit equivariant group cohomology following \cite[Section 1]{In05}. For $N$  a $\Gamma$-equivariant $kG$-module we denote by $\mathrm{H}_{\Gamma}^n(G,N)$ the $\Gamma$-equivariant $n$-th cohomology of $G$ with coefficients in $N.$ Inspired by \cite[Section 2]{Gou}, we present in Subsection \ref{subsec22}  a definition of $\Gamma$-equivariant Hochschild cohomology. Let $A$ be a finite dimensional $\Gamma$-algebra (see the first paragraphs of Subsection \ref{subsec22}) and let $M$ be a $\Gamma$-equivariant $(A-A)$-bimodule. Equivalently, $M$ is an $(A-A)$-bimodule such that the  linear action of $\Gamma$ on $M$ is compatible with the left, respectively right, $A$-module structure of $M.$  The $\Gamma$-equivariant $n$-th   Hochschild cohomology will be denoted by $\HH_{\Gamma}^n(A,M).$ Historically, it seems that this variant of $\Gamma$-equivariant Hochschild cohomology  of an algebra with coefficients in a bimodule was previously defined by K. K. Jensen in the context of Banach algebras, \cite[Definition I.1.7]{Je}. In a recent article \citep{KoPi}, the authors develop another variant still  called equivariant Hochschild cohomology which, in the above context, is denoted by $\Hrm^*_{\Gamma}(A,M).$ This is the homology of the total complex of a bicomplex, obtained by mixing the Hochschild cochain complex and the cochain complex of the $\Gamma$-group cohomology \cite[Definition 5]{KoPi}. The objective of  Subsection \ref{subsec23} is to recall some basic aspects of the relative homological algebra, which shall be need it for proving the second main result of this paper.

The first Hochschild cohomology of an algebra is the same as the space of outer derivations on this algebra.  Studying the triviality, or non-triviality, of outer derivations on algebraic structures is an important topic in a variety of areas: derivations on Lie algebras \cite{Ja}, derivations on group algebras \cite{Ar}, derivations on block algebras of finite groups \cite[Section 7]{Li22}, \cite{Br}, etc. Let $x\in G$ and denote by $\Gamma_x$ the stabilizer subgroup of $x$ in $\Gamma.$  In the first main result we study the above problem for the space of $\Gamma_x$-equivariant outer derivations of the group algebra $kG$. In Section \ref{sec3}, as a particular case of Proposition \ref{prop33} we obtain a $k$-linear injective map from  $\Hrm_{\Gamma_x}^1(C_G(x),k)$ into the
first $\Gamma_x$-equivariant Hochschild cohomology of $kG.$  Using this, in the first main theorem of this paper we obtain some necessary conditions for the non-triviality of the first $\Gamma_x$-equivariant Hochschild cohomology of $kG.$ In Section \ref{sec4} we introduce these conditions, by  defining  $C_{\Gamma,G}(p)$-good elements (Definition \ref{def41}), which are applied  for a finite $\Gamma$-group $G,$ when $p$ divides $|G|$ and $k$ is a field of characteristic $p.$ We begin Section \ref{sec4} with some examples of $C_{\Gamma,G}(p)$-good elements and finish by giving the proof of the following result.

\begin{theorem}\label{thm11}
Let $k$ be a field of characteristic $p$ and $G$ be a $\Gamma$-group.
 If there is a $C_{\Gamma,G}(p)$-good element $x$ in $G$ then $\HH_{\Gamma_x}^1(kG)\neq 0.$
\end{theorem}
Let $X$ be a system of representatives of conjugacy classes in $G.$ For group algebras it is well known that we have an additive decomposition of the  Hochschild cohomology
    $$\HH^n(kG)\cong \bigoplus_{x\in X}\Hrm^n(C_G(x),k),$$
where $kG$ is the group algebra and $k$ is the trivial $kC_G(x)$-module. This decomposition goes back to Burghelea \citep{Bu}. Section \ref{sec3} more precisely Proposition \ref{prop33}, can be viewed  as a first step to obtain a similar decomposition for equivariant Hochschild cohomology; this is left for a future project.

For the second main result, we continue  with notations of Subsection \ref{subsec22}, that is $A$ a $\Gamma$-algebra and $M$ is a $\Gamma$-equivariant $(A-A)$-bimodule. The objective of the first part of Section \ref{sec6} is to present an alternative approach to the $\Gamma$-equivariant Hochschild cohomology in a similar way as suggested by Inassaridze in \cite[Section 5]{In}; more precisely, we reinterpret Subsection \ref{subsec22} as in \cite[Definition 5.1]{In}, but for the cohomological case. The action of $\Gamma$ on $C^*(A,M)$ (where $C^*(A,M)$ is the Hochschild cochain complex of $A$ with coefficients in $M,$ see (\ref{eq22})) is induced by the action of $\Gamma$ on $A$ and on $M$ and, in this case the $\Gamma$-equivariant Hochschild cohomology is
$$\HH_{\Gamma}^*(A,M)=\Hrm_{\Gamma}^*(C^*(A,M)),$$
see Definition \ref{defn61}. The enveloping algebra $A^e$ of $A$ (i.e. $A^e=A\otimes A^{op}$) remains a $\Gamma$-algebra by
$$(a\otimes b)(c\otimes d)=ac \otimes db,\quad {}^{\sigma}(a\otimes b)={}^{\sigma}a\otimes {}^{\sigma}b,$$
 for any $ a,b,c,d\in A, \sigma \in \Gamma.$  The enveloping algebra of the $\Gamma$-algebra $A$  is introduced in \citep{KoPi} in the context of the so-called "oriented" algebras and is denoted by $(A,\Gamma)^e.$ As a vector space $(A,\Gamma)^e$ is $A\otimes k\Gamma\otimes A$, with the multiplication
$$(a\otimes \sigma \otimes b)\cdot (c\otimes \tau \otimes d)=a({}^{\sigma}c)\otimes \sigma \tau \otimes ({}^{\sigma}d) b,$$
 for any $a,b,c,d\in A, \sigma,\tau \in \Gamma.$ From this reference \citep{KoPi} we shall use the notations and some results; we are not interested in the oriented aspect, only in associative $G$-algebras (with trivial orientation). In Lemma \ref{lem71} (i) we will see that $(A,\Gamma)^e\cong A^e \star \Gamma,$ as $\Gamma$-algebras; here $A^e \star \Gamma$ is the skew group algebra of the $\Gamma$-algebra $A^e$ with the multiplication recalled in (\ref{eq61}). Thus,  we identify  $k\Gamma$ with a subalgebra of $A^e\star \Gamma.$ This identification is given by
$$\sigma \mapsto (1_A\otimes 1_A)\star\sigma,\quad \sigma\in \Gamma.$$
Using Subsection \ref{subsec23} we define the relative $\Ext$ in this case by $\Ext_{(A,\Gamma)^e,k\Gamma}^*(A,M).$

The first statement of the second main theorem says that the $\Gamma$-equivariant Hochschild cohomology of $\Gamma$-algebras with coefficients in a $\Gamma$-equivariant bimodule admits an interpretation as a $k\Gamma$-relative $\Ext.$ In the second statement we give a natural map between the two types of $\Gamma$-equivariant Hochschild cohomology (as defined in \cite{KoPi} and the one recalled in Subsection \ref{subsec22}). We also establish that these two types of $\Gamma$-equivariant Hochschild cohomology are the same if $k\Gamma$ is a separable $k$-algebra.

\begin{theorem}\label{thm12}
Let $A$ be a $\Gamma$-algebra and $M$ be a $\Gamma$-equivariant $(A-A)$-bimodule.
\begin{itemize}
\item[i)] There is a $k$-linear isomorphism
$$\HH_{\Gamma}^n(A,M)\cong \Ext_{(A,\Gamma)^e,k\Gamma}^n (A,M).$$
\item[ii)] There is a natural map $$\HH_{\Gamma}^n(A,M) \rightarrow \Hrm_{\Gamma}^n(A,M),
$$ which is a $k$-linear isomorphism if $k\Gamma$ is a separable $k$-algebra.
\end{itemize}
\end{theorem}
In the second part of Section \ref{sec6} we prove the above theorem.  We shall use the notations: $G^{\times n},$ which means $G\times \cdots \times G$ where $G$ appears $n$ times; $A^{\otimes n},$ which means $A\otimes \cdots \otimes A$ where $A$ appears $n$ times and $\otimes$ is $\otimes_k$.

\section{Reminder on  equivariant cohomology of groups, equivariant Hochschild cohomology and relative homological algebra.}\label{sec2}
Let $\Gamma$ and $G$ be finite groups, $k$ be a field,  and let $n$ be a non-negative integer.
\subsection{\texorpdfstring{\textit{ The  $\Gamma$-equivariant cohomology of groups.}}{The Gamma-equivariant cohomology of groups.}}\label{subsec21}

Let $N$ be a $\Gamma$-equivariant $kG$-module. The $\Gamma$-\textit{equivariant} $n$-th cohomology of $G$ with coefficients in $N$ is
\[\mathrm{H}_{\Gamma}^n(G,N)=\mathrm{H}^n(C_{\Gamma}^*(G,N))=\Ker \partial^n/\Ima \partial^{n-1},\]
where the cochain complex $C_{\Gamma}^*(G,N)$ is given by
\begin{equation}\xymatrix{&C_{\Gamma}^0(G,N)\ar[r]^{\partial^0}&\ldots \ar[r]&C_{\Gamma}^{n-1}(G,N)\ar[r]^{\partial^{n-1}}\ar[r]&C_{\Gamma}^n(G,N)\ar[r]^{\partial^n}
\ar[r]&C_{\Gamma}^{n+1}(G,N)\ldots }
\end{equation}
and  $C_{\Gamma}^n(G,N)$ is the $k$-space of all $\Gamma$-maps $\varphi:G^n\rightarrow N.$ Clearly $C_{\Gamma}^n(G,N)$ is the $k$-subspace of $\Gamma$-invariant maps of $C^n(G,N),$ where $C^n(G,N)$ is in fact the $k$-space $\mathrm{Maps}(G^n,N)$ and, the action of $\Gamma$ is $$\Gamma\times C^n(G,N)\rightarrow C^n(G,N) \quad (\sigma,f)\mapsto {}^{\sigma}f;\quad {}^{\sigma}f:G^{\times n}\rightarrow N$$
          $$ {}^{\sigma}f(g_1,\dots,g_n)={}^{\sigma}(f({}^{\sigma^{-1}}g_1,\dots,{}^{\sigma^{-1}}g_n)),\quad n \ge 1.$$ In the above, the differentials are defined by
\[ \partial^n:C_{\Gamma}^n(G,N)\to C_{\Gamma}^{n+1}(G,N)\]
\[(\varphi:G^{\times n}\to N )\mapsto (\partial^n(\varphi):G^{\times (n+1)}\to N)\]
\begin{align*}
\partial^n(\varphi)(g_1,\ldots , g_{n+1})&=g_1 \varphi(g_2, \ldots , g_{n+1})+
\sum_{j=1}^n(-1)^j \varphi(g_1, \ldots , g_jg_{j+1}, \ldots, g_{n+1})\\
&+(-1)^{n+1}\varphi (g_1,\ldots, g_n),
\end{align*}
for any $g_1, \ldots, g_{n+1}\in G.$ If $n=0$ then $$C_{\Gamma}^0(G,N)=N^{\Gamma}=\{m\in N \mid {}^{\sigma}m=m,\mbox{ for any }\sigma\in \Gamma\}.$$

\subsection{\textit{The equivariant Hochschild cohomology.}}\label{subsec22}

Let $M$ be a $k\Gamma$-module and $A$ be a $k$-algebra, such that there exists an action of $\Gamma$ on $A$ by $k$-automorphisms. This means that $A$ is a $\Gamma$-algebra with the $\Gamma$-action  $\Gamma \times A\rightarrow A$ denoted by $$(\sigma,a)\mapsto {}^{\sigma}a, \quad  (\sigma,a)\in \Gamma \times A.$$ The left  linear action of $\Gamma$ on $M$ is denoted by $$(\sigma,m)\mapsto {}^{\sigma} m,  \quad (\sigma,m)\in \Gamma \times M.$$
The same notation for the action of $\Gamma$ on $M$ is used as in the case of the action on $A$, but the reader should understand the difference from the context.
 We assume that $M$ is a $\Gamma$-equivariant $(A-A)$-bimodule. With the above notations, using  \cite[Section 2]{Gou}, we define  the $n$-th $\Gamma$-\textit{equivariant Hochschild cohomology}, that is
$$\HH_{\Gamma}^n(A,M)=\mathrm{H}^n(C_{\Gamma}^*(A,M))=\Ker  \delta^n/\Ima \delta^{n-1},$$
where $C_{\Gamma}^*(A,M)$ is a cochain complex given by

$$ C_{\Gamma}^n(A,M)=\{\varphi\in C^n(A,M)\mid \ \varphi({}^{\sigma}a_1 \otimes \dots \otimes^{\sigma}a_n)={}^{\sigma}(\varphi(a_1\otimes\dots\otimes a_n)), \mbox{ for any } \sigma\in \Gamma\}.$$
Recall that $C^n(A,M)$ is the $n$-th  space Hochschild cochains of the algebra $A$ with coefficients in the bimodule $M,$ defining the classical  Hochschild  cochain complex $C^*(A,M):$

\begin{equation}\label{eq22}\xymatrix{&C^0(A,M)\ar[r]^{\delta^0}&\ldots \ar[r]&C^{n-1}(A,M)\ar[r]^{\delta^{n-1}}\ar[r]&C^n(A,M)\ar[r]^{\delta^n}\ar[r]&C^{n+1}(A,M)\ldots}
\end{equation}
where $$C^n(A,M)=\{f\mid f:A^{\otimes n}\rightarrow M,\mbox{ } f \mbox{ is } k-\mbox{linear}\},$$ and $\delta^n$ is defined below. Clearly $C_{\Gamma}^n(A,M)$ consists of all $n$-cochains, that are $\Gamma$-equivariant, and it is a $k$-subspace of $C^n(A,M).$ By the easy implication $\varphi\in C_{\Gamma}^n(A,M)$ then $\delta^n(\varphi)\in C_{\Gamma}^{n+1}(A,M),$ it follows that $\{C_{\Gamma}^n(A,M),\delta^n\}_{n\geq 0}$ is a cochain subcomplex of the Hochschild cochain complex. For $n=0$ the space $C_{\Gamma}^0(A,M)$ is identified with $M^{\Gamma}$.
Explictly, the cochain complex $C_{\Gamma}^*(A,M)$ is
\begin{equation}\xymatrix{&C_{\Gamma}^0(A,M)\ar[r]^{\delta^0}&\ldots \ar[r]&C_{\Gamma}^{n-1}(A,M)\ar[r]^{\delta^{n-1}}\ar[r]&C_{\Gamma}^n(A,M)\ar[r]^{\delta^n}\ar[r]&C_{\Gamma}^{n+1}(A,M)\ldots}
\end{equation}
\[ \delta^n:C_{\Gamma}^n(A,M)\to C_{\Gamma}^{n+1}(A,M)\]
\[(\varphi:A^{\otimes n}\to M )\mapsto (\delta^n(\varphi):A^{\otimes (n+1)}\to M)\]
\begin{align*}
\delta^n(\varphi)(a_1\otimes\ldots \otimes a_{n+1}) &= a_1\cdot \varphi(a_2\otimes \ldots \otimes a_{n+1}) \\
&\quad + \sum_{j=1}^n(-1)^j \varphi(a_1\otimes\ldots \otimes a_j\cdot a_{j+1}\otimes \ldots \otimes a_{n+1}) \\
&\quad + (-1)^{n+1} \varphi (a_1\otimes\ldots \otimes a_n)\cdot a_{n+1},
\end{align*}
for all $a_1, \ldots, a_{n+1}\in A$, if $n>0$.

In the case $n=0$ the differential $\delta^0:M\to C^{1}(A,M)$ is given by
$$\delta^0(m)(a_1)=a_1\cdot m-m\cdot a_1$$ for any $m\in M, a_1\in A.$ It follows that
$$\HH_{\Gamma}^0(A,M)\cong\{m\in M^{\Gamma} | am=ma\  \text{for all}\ a\in A \}.$$

\subsection{\textit{Relative homological algebra.}}\label{subsec23}
In this subsection we shall use the following notations: $k$ is a field, $A$ is a $k$-algebra and, $B$ is  a $k$-subalgebra of $A.$ The main source goes back to Hochschild \cite[Section 2]{Ho}. An exact sequence of $A$-modules
$$
\xymatrix{
& \ldots \ar[r]
& M_{i+1} \ar[r]^{d_{i+1}}
& M_i \ar[r]^{d_i}
& M_{i-1} \ar[r]
& \ldots
}$$
is $(A,B)$-exact if $\Ker d_i$ is a direct summand (as $B$-module) of $M_i$, $i\in \Z.$ Clearly, a chain complex $(M_*,d_*)$ is $(A,B)$-exact, if there is a sequence of $B$-module homomorphisms $h_i:M_i\rightarrow M_{i-1}$, $i \in \Z$ giving a $B$-homotopy. Some $A$-module $P$ is relative $B$-projective ( also called $(A,B)$-projective) if for all $(A,B)$-exact sequences
$$
\xymatrix{
& 0 \ar[r]
& K \ar[r]^{i}
& L \ar[r]^{\pi}
& M \ar[r]
& 0
}$$ and for all $A$-module homomorphisms $g:P\rightarrow M$ there is $g':P\rightarrow L$ an $A$-module homomorphism such that $g=\pi \circ g'.$
A relative $B$-projective resolution $P_{*}$ of an $A$-module $M$ is an $(A,B)$-exact sequence
$$\dots \rightarrow P_{n+1}\rightarrow P_n\rightarrow P_{n-1}\rightarrow \dots \rightarrow P_1\rightarrow P_0\rightarrow M\rightarrow 0,$$
with each $P_i,$ $i\in \N$ being relative $B$-projective. Using a Comparison Theorem (which is known to exist in relative homological algebra) the relative derived left functor is defined in the following way: let $M$, $N$ be two $A$-modules and $${}_AP_{*}\rightarrow {}_AM\rightarrow 0$$ be a relative $B$-projective resolution of $M$. We obtain the cochain complex
$$0\rightarrow\Hom_A(M,N)\rightarrow\Hom_A(P_*,N)$$
and, the $B$-relative $\Ext$ is $\Ext_{A,B}^*(M,N),$ given by
$$
\Ext_{A,B}^n(M,N)
= \Ker(\Hom_A(d_{n+1},N)) / \operatorname{Im}(\Hom_A(d_n,N)),
\quad n \in \mathbb{N}
$$
where
$$
P_{n+1} \xrightarrow{d_{n+1}} P_n \xrightarrow{d_n} P_{n-1}
$$
are terms from the above relative $B$-projective resolution.
When $B=k$ relative $k$-projectivity is the same as $A$-projectivity and $\Ext_{A,k}^*(M,N)$ is the classical $\Ext_A^*(M,N).$
If $B'$ and $B$ ar $k$-subalgebras of $A,$ such that $B'\subseteq B,$ then there is a natural map, for any $n\in \N,$
\begin{equation} \label{eqSha1}
\Ext_{A,B}^n(M,N)\rightarrow \Ext_{A,B'}^n(M,N).
\end{equation}
Gerstenhaber and Shack \cite[Theorem 1.2]{GeSc} show that the natural map
\begin{equation} \label{eqSha2}
\Ext_{A,B}^*(M,N)\rightarrow \Ext_{A}^*(M,N)
\end{equation}
is an isomorphism, if $B$ is $k$-separable.

\section{An emebdding into equivariant Hochschild cohomology of group algebras.}\label{sec3}
We continue with the notations used in the Introduction. Let $n\geq 0, n\in \Z$ and $x\in G.$ For the additive decomposition $\HH^n(kG)$ in a direct sum of group cohomology with trivial coefficients, there is a recent explicit description \citep{LiZh}, and we shall use the maps presented in \cite[Section 5]{CoTo}.
 We have:
\begin{align}\label{eq31}
    \pi_{G,x}^n: C^n(kG,kG)\rightarrow C^n(C_G(x),kx) \nonumber \\ (\varphi:(kG)^{\otimes n}\rightarrow kG)\mapsto (\pi_{G,x}^n(\varphi):(C_G(x))^{\times n}\rightarrow kx) \nonumber \\  \pi_{G,x}^n(\varphi)(h_1,h_2,\dots,h_n)=a_{1,x}x,
    \end{align}
where $a_{1,x}$ is the coefficient of $x$ in $\varphi(h_1\otimes \dots \otimes h_n)h_n^{-1}h_{n-1}^{-1} \dots  h_1^{-1},$
for any $h_1,h_2,\dots,h_n\in C_G(x).$

  \begin{align}\label{eq32}\nu_{G,x}^n:C^n(C_G(x),kx)\rightarrow C^n(kG,kG) \nonumber \\
  (\psi:(C_G(x))^{\times n}\rightarrow kx)\mapsto (\nu_{G,x}(\psi):(kG)^{\otimes n}\rightarrow kG) \nonumber \\
   \nu_{G,x}^n(\psi)(g_1\otimes \dots \otimes g_n)=\sum_{j=1}^{n_x}\psi(h_{j,1},\dots,h_{j,n})x^{-1}x_j g_1 \dots g_n,
   \end{align}
for any $g_1,\ldots,g_n\in G,$ and $\psi\in C^n(C_G(x),kx).$ Let $n_x$ be the cardinal of the conjugacy class $C_x=\{^gx\mbox{ }|\mbox{ }g\in G\}.$ In the above lines $h_{j,1},\ldots,h_{j,n}$ are determined by the elements $g_1,g_2,\ldots,g_n,$ and $x_1,x_2,\ldots, x_{n_x}$ in the following way: consider   the coset decomposition
   $$G=C_G(x)\gamma_{1,x}\cup \dots \cup C_G(x)\gamma_{n_x,x}=\gamma_{1,x}^{-1} C_G(x)\cup \dots \cup \gamma_{n_x,x}^{-1} C_G(x),$$
such that $$C_x=\{x, \gamma_{2,x}^{-1}x\gamma_{2,x},\ldots,\gamma_{n_x,x}^{-1}x\gamma_{n_x,x}\}.$$
We adopt the notation $x_j=\gamma_{j,x}^{-1}x\gamma_{j,x},$ for any $j\in \{1,\ldots,n_x\}$ and $\gamma_{1,x}=1.$ For each $j\in \{1,\ldots,n_x\}$ we obtain:

  \begin{equation}\label{eq33}
  \gamma_{j,x}g_1=h_{j,1}\gamma_{s_j^1,x};\  \gamma_{s_j^1,x}g_2=h_{j,2}\gamma_{s_j^2,x};\ \dots\ ;\gamma_{s_j^{n-1},x}g_n=h_{j,n}\gamma_{s_j^n,x};
  \end{equation}
where $s_j^1,\ldots,s_j^n\in\{1,\ldots,n_x\}.$
 We shall use the notation $\gamma_x=\{\gamma_{1,x},\ldots,\gamma_{n_x,x}\}$ for the above system of representatives of left cosets of $C_G(x)$ in $G,$ with  $\gamma_x^{-1}=\{\gamma_{1,x}^{-1},\ldots,\gamma_{n_x,x}^{-1}\}$ being  a system of representatives of right cosets of $C_G(x)$ in $G.$ Note that the definition of $\nu_{G,x}^n$ is dependent on $\gamma_x.$

For any $\sigma \in\Gamma_x,$ since ${}^{\sigma}C_G(x)=C_G(x)$ and $^{\sigma}C_x=C_x,$ it follows that ${}^{\sigma} \gamma_x:=\{{}^{\sigma}\gamma_{1,x},\ldots,{}^{\sigma}\gamma_{n_x,x}\}$ is a system of representatives of left cosets of $C_G(x)$ in $G,$ such that
    $$G={}^{\sigma}(\gamma_{1,x}^{-1}) C_G(x)\cup \dots \cup {}^{\sigma}( \gamma_{n_x,x}^{-1}) C_G(x),$$ and
    $$C_x=\{x, {}^{\sigma}(\gamma_{2,x}^{-1})x{}^{\sigma}(\gamma_{2,x}),\ldots,{}^{\sigma}(\gamma_{n_x,x}^{-1})x{}^{\sigma}(\gamma_{n_x,x})\}.$$
    Moreover, for each $j\in \{1,\ldots,n_x\},$ by applying $\sigma$ on the formulas in (\ref{eq33}), we obtain:

  \begin{equation}\label{eq33'} {}^{\sigma}\gamma_{j,x}\mbox{ }^{\sigma}g_1={}^{\sigma}h_{j,1}\mbox{ }^{\sigma}\gamma_{s_j^1,x}; {}^{\sigma}\gamma_{s_j^1,x}\mbox{ }^{\sigma}g_2=
    {}^{\sigma}h_{j,2}\mbox{ }^{\sigma}\gamma_{s_j^2,x};
    \dots;{}^{\sigma}\gamma_{s_j^{n-1},x}\mbox{ }^{\sigma}g_n={}^{\sigma}h_{j,n}\mbox{ }^{\sigma}\gamma_{s_j^n,x}.
   \end{equation}

\begin{proposition}\label{prop32} Let $G$ be a finite $\Gamma$-group and $x\in G$. In particular $kG$ becomes a $\Gamma$-algebra and we consider $kx$ as a trivial  $\Gamma_x$-equivariant $kG$-module.  The following statements hold:
\begin{itemize}
\item[i).] $\pi_{G,x}^n(C_{\Gamma_x}^n(kG,kG))\subseteq C_{\Gamma_x}^n(C_G(x),kx);$
\item[ii).]  If there is $\gamma_x$ satisfying (\ref{eq33}) such that $\gamma_x$ is $\Gamma_x$-stable then $\nu_{G,x}^n(C_{\Gamma_x}^n(C_G(x),kx))\subseteq C_{\Gamma_x}^n(kG,kG).$
\end{itemize}
\end{proposition}

\begin{proof}

\begin{enumerate}
\item[i).] We consider $$\varphi\in C_{\Gamma_x}^n(kG,kG),\sigma\in \Gamma_x,\quad h_1,\ldots,h_n\in C_G(x),$$
Then $${}^{\sigma}(\pi_{G,x}^n(\varphi)(h_1,\ldots,h_n))={}^{\sigma}(a_{1,x}x)=a_{1,x}x,$$
since  $kx$ is a trivial $\Gamma_x$-equivariant, $C_G(x)$-module.

The element $\pi_{G,x}^n(\varphi)(^{\sigma}h_1,\dots,^{\sigma}h_n)$ is the coefficient of $x$ in $$\varphi(^{\sigma}h_1\otimes\dots \otimes^{\sigma}h_n)^{\sigma}(h_n^{-1})\cdot \dots \cdot ^{\sigma}(h_1^{-1})=^{\sigma}(\varphi(h_1\otimes\dots \otimes h_n)h_n^{-1}\cdot \dots \cdot h_1^{-1})=a_{1,x}x,$$ since $\sigma\in \Gamma_x.$

\item[ii).] Let $\psi \in C_{\Gamma_x}^n(C_G(x),kx),$ $g_1,\dots,g_n\in C_G(x)$ and $\sigma \in \Gamma_x.$ Since $\gamma_x$ is $\Gamma_x$-stable it follows that the set ${}^\sigma\gamma_x$ is a permutation of $\gamma_x.$ Then, for each $j\in\{1,\ldots, n_x\}$ we have ${}^{\sigma}\gamma_{j,x}=\gamma_{i,x},$ where $i\in\{1,\ldots,n_x\}.$ It follows that the formulas (\ref{eq33'}) change to
$$\gamma_{i,x}\mbox{ }^{\sigma}g_1={}^{\sigma}h_{j,1}\gamma_{t_i^1,x}; \gamma_{t_i^1,x}\mbox{ }^{\sigma}g_2=
    {}^{\sigma}h_{j,2}\gamma_{t_i^2,x};
    \dots;\gamma_{t_i^{n-1},x}{ }^{\sigma}g_n={}^{\sigma}h_{j,n}\gamma_{t_i^n,x},$$
 where $t_i^1,\ldots,t_i^n\in\{1,\ldots,n_x\}.$
 The above correspondence $j\leftrightarrow i$ is bijective, where $i,j\in\{1,\ldots,n_x\};$  and, for each $i\in\{1,\ldots,n_x\},$ similarly to (\ref{eq33})  we have
 $$\gamma_{i,x}\mbox{ }^{\sigma}g_1=h'_{i,1}\gamma_{t_i^1,x}; \gamma_{t_i^1,x}\mbox{ }^{\sigma}g_2=
    h'_{i,2}\gamma_{t_i^2,x};
    \dots;\gamma_{t_i^{n-1},x}{ }^{\sigma}g_n=h_{i,n}\gamma_{t_i^n,x},$$
with the same $t_i^1,\ldots,t_i^n\in\{1,\ldots,n_x\}$ and $h'_{i,1},\ldots, h'_{i,n}\in C_G(x).$ It follows that when $j$ runs in $\{1,\ldots,n_x\}$ then $i$ runs in $\{1,\ldots,n_x\}$ and
$${}^{\sigma}h_{j,1}=h'_{i,1}, \ldots, {}^{\sigma}h_{j,n}=h'_{i,n}.$$
Next, from the above conditions and using (\ref{eq32}), we obtain that:
\begin{align*}
{}^{\sigma}(\nu_{G,x}^n(\psi)(g_1\otimes \dots \otimes g_n))&={}^{\sigma}\left(\sum_{j=1}^{n_x}\psi(h_{j,1},\dots,h_{j,n})x^{-1} x_j g_1 \cdot \dots \cdot g_n\right)\\
&=\sum_{j=1}^{n_x}\psi({}^{\sigma}h_{j,1},\dots,{}^{\sigma}h_{j,n})x^{-1}({}^{\sigma} x_j)  { }^{\sigma} g_1 \cdot \dots \cdot {}^{\sigma}g_n\\
     &=\sum_{i=1}^{n_x}\psi(h'_{i,1},\dots,h'_{i,n})x^{-1} x_i { }^{\sigma} g_1 \cdot \dots \cdot {}^{\sigma}g_n\\
     &=\nu_{G,x}^n(\psi)({}^{\sigma}g_1\otimes \dots \otimes {}^{\sigma} g_n),
 \end{align*}

since $\psi \in C_{\Gamma_x}(C_G(x),kx)$ and ${}^{\sigma}x_j=x_i.$
\end{enumerate}

\end{proof}
By Proposition \ref{prop32} we define  next the restrictions:
\begin{equation} \label{eq34}
\pi_{\Gamma_x,G}^n:C_{\Gamma_x}(kG,kG)\rightarrow C_{\Gamma_x}(C_G(x),k),
\end{equation}
 where $\pi_{\Gamma_x,G}^n$ is the  restriction of $\pi_{G,x}$ to $C_{\Gamma_x}(kG,kG);$
\begin{equation}  \label{eq35}
\nu_{\Gamma_x,G}^n:C_{\Gamma_x}(C_G(x),k)
     \rightarrow C_{\Gamma_x}(kG,kG),
\end{equation}
where $\nu_{\Gamma_x,G}$ is the restriction of $\nu_{G,x}^n$ to $C_{\Gamma_x}(C_G(x),k),$
if there is $\gamma_x$ satisfying (\ref{eq33}) such that $\gamma_x$ is $\Gamma_x$-stable.
\begin{proposition}\label{prop33} Let $n\in \mathbb{Z},$ $n\geq 0$ and $x\in G.$ We identify $kx$ with $k$ as trivial $\Gamma_x$-equivariant $kC_G(x)$-module. If there is $\gamma_x$ satisfying (\ref{eq33}) such that $\gamma_x$ is $\Gamma_x$-stable then there is an injective $k$-linear map of cohomology spaces
    $$\nu_{\Gamma_x,G}^n :  \Hrm_{\Gamma_x}^n(C_G(x),k)
     \rightarrow \HH_{\Gamma_x}^n(kG),$$
with its left inverse
    $$\pi_{\Gamma_x,G}^n:\HH_{\Gamma_x}^n(kG)\rightarrow \Hrm_{\Gamma_x}^n(C_G(x),k).$$
 \end{proposition}
\begin{proof}
This statement follows by applying  (\ref{eq34}), (\ref{eq35})  and the same computations as in the proofs of \cite[Lemma 5.2 a), d), e)]{CoTo}, see also \cite[Remark 5.4]{CoTo}.
\end{proof}

\section{\texorpdfstring{$C_{\Gamma,G}(p)$-good elements and proof of Theorem \ref{thm11}.}{C(Gamma,G,p)-good elements and proof of Theorem 1.1.}}\label{sec4}
Let  $\Omega$  be a finite group and $(H,\cdot)$ be a finite $\Omega$-group. The normal subgroup $\Omega H$ of $H,$ generated by  the set
$$[\Omega H]=\{h\cdot {}^{\sigma}(h^{-1})\mid h\in H, \sigma\in \Omega\}$$ is considered by Inassaridze in \citep{In}.
We consider the subgroup $[H,\Omega]$ of $H$ generated by $[\Omega H].$ In general,
$[H,\Omega]$ (that is  a subgroup of $\Omega H$) is not a normal subgroup of $H.$  Note that if $x\in H,$ then $\Omega_x$ acts on $C_H(x).$ We just mention the following easy, probably well-known result, giving a case for which the above subgroups are the same.
\begin{lemma}\label{lem41}
Let $H$ be a finite $\Omega$-group such that $(H,\cdot)$ is a subgroup of $(\Omega,\cdot)$ satisfying the property that if   $\sigma\in H,$ $h\in H,$ then $^{\sigma}h=\sigma\cdot h\cdot \sigma^{-1}.$ Then $$H'\leq [H,\Omega]=\Omega H.$$
\end{lemma}
\begin{proof}
The first inclusion $H'\leq [H,\Omega]$ is clear. Let $g\in H, h\in H, \sigma\in \Omega.$ An element of the form $g\cdot (h\cdot ^{\sigma}(h^{-1}))\cdot g^{-1}$ in the generating set of $\Omega H$ can be written
$$g\cdot (h\cdot ^{\sigma}(h^{-1}))\cdot g^{-1}={}^gh\cdot^g(^{\sigma}(h^{-1}))={}^gh\cdot ^{g\sigma}(h^{-1})={}^gh\cdot ^{g\sigma g^{-1}}(^gh^{-1}),$$
hence  $g \cdot (h\cdot ^{\sigma}(h^{-1}))\cdot g^{-1} \in [H,\Omega].$
\end{proof}
The following definition is inspired by the Commutator index property $C(p)$ \cite[Deefinition 1.1]{FJL}.
\begin{definition} \label{def41}
Let $G$ be a finite $\Gamma$-group such that $p$ divides $|G|.$ We say that the element $x\in G$ is $C_{\Gamma,G}(p)$-\textit{good} if:  $p$ divides the index $|C_{G}(x):\Gamma_xC_G(x)|,$ the derived subgroup $(C_G(x))'$ is a subgroup of $\Gamma_xC_G(x)$ and, there is $\gamma_x$ a system of representatives of left cosets of $C_G(x)$ in $G$ satisfying (\ref{eq33}), such that $\gamma_x$ is $\Gamma_x$-stable.
\end{definition}

\begin{example}\label{ex1}\quad

\begin{itemize}

\item[1)] Let $G$ be a finite group such that  $p$ divides $|G/G'|$ and assume that $\Gamma=G$ is acting by conjugation on $G$. Using Lemma \ref{lem41}, for any  $x\in G,$ we have that $\Gamma_x=C_G(x)$ and   $$\Gamma_xC_G(x)=[C_G(x),C_G(x)]=(C_G(x))'.$$ In this case any $x\in Z(G)$ is a $C_{\Gamma,G}(p)$-good element.

\item[2)] Let $(\Gamma,\cdot)=(\mathrm{GL}_2(\mathbb{Z}_2),\cdot),$  $(G,+)=(\mathbb{Z}_2^2,+)$ and we assume $p=2.$ We consider that $\Gamma$ is acting on $G$ by $^Av=A\cdot v,$ for every $A\in \mathrm{GL}_2(\mathbb{Z}_2)$ and $v\in G.$

We choose $x=\begin{pmatrix}1 \\ 0\end{pmatrix}.$ Then
\begin{align*}
\Gamma_x=\left\{\begin{pmatrix}a_1 & a_2\\a_3 & a_4\end{pmatrix}\in \mathrm{GL}_2(\mathbb{Z}_2)\mid \begin{pmatrix}a_1 & a_2\\a_3 & a_4\end{pmatrix}\cdot\begin{pmatrix}1 \\ 0\end{pmatrix}=\begin{pmatrix}1 \\ 0\end{pmatrix}\right\}=\left\{\begin{pmatrix}1 & 0\\0 & 1\end{pmatrix},\begin{pmatrix} 1& 1\\0 & 1\end{pmatrix}\right\}.
\end{align*}

Since $G$ is abelian $C_G(x)=G$ and $\Gamma_x G=[G,\Gamma_x].$   We obtain

\begin{align*}
[G,\Gamma_x]
&=\langle v+A\cdot(-v)\mid v\in \mathbb{Z}_2^2,A\in \Gamma_x\rangle \\
&=\left\langle \begin{pmatrix}0 \\ 0\end{pmatrix},\begin{pmatrix}x_1 \\ x_2\end{pmatrix}+\begin{pmatrix}x_1+x_2 \\ x_2\end{pmatrix}\mid x_1,x_2\in \mathbb{Z}_2\right\rangle \\
&=\left\{\begin{pmatrix}0 \\ 0\end{pmatrix},\begin{pmatrix}1 \\ 0\end{pmatrix}\right\}.
\end{align*}

Therefore  $|G:\Gamma_x G|=2$ and, moreover $$ G'=\{1\} \leq \Gamma_x G.$$
Thus $\begin{pmatrix}1 \\ 0\end{pmatrix}$ is $C_{\Gamma,G}(2)$-good, since $\gamma_x=\{1\}.$

\end{itemize}
\end{example}
\begin{remark} It is clear that Definition \ref{def41} contains a lot of conditions to be verified, but we see in Example \ref{ex1} that it is not that difficult to obtain examples  for which Definition \ref{def41} applies. This remark is also a good opportunity to launch the following question: can someone find a classification of all finite $\Gamma$-groups $G$ with $p$ dividing the order of $G,$ which contains at least one $C_{\Gamma,G}(p)$-good element? If $\Gamma$ acts trivially on $G,$ there are many examples of finite groups $G$ with $C_{\Gamma,G}(p)$-good element, but in this case $\HH^1_{\Gamma}(kG)$ is $\HH^1(kG).$
\end{remark}
\begin{proof}\textbf{(of Theorem \ref{thm11}.)}
 By our hypothesis, it follows that there is $x\in G$ such that $p$ divides $|C_G(x):\Gamma_xC_G(x)|$ and
$$(C_G(x))'\leq \Gamma_xC_G(x)\unlhd C_G(x).$$
Since $C_G(x)/\Gamma_xC_G(x)$ is an abelian group, with order divisible by $p,$ it follows that there is a subgroup $\mathbf{A}/\Gamma_xC_G(x)$ in
$C_G(x)/\Gamma_xC_G(x),$ of order $p;$ hence there is $\mathbf{A}$ a normal subgroup of $C_G(x)$ of index $p,$ such that $\Gamma_xC_G(x)\leq \mathbf{A}.$
Equivalently, we have that $\Hrm_{\Gamma_x}^1(C_G(x),\mathbb{F}_p)\neq 0.$
Explicitly, there is a surjective  group homomorphism $f:(C_G(x),\cdot)\rightarrow (\mathbb{F}_p,+),$ with $f(g_0)\neq 0,$ for some $g_0\in C_G(x)\setminus \{1\},$ such that
$\Gamma_xC_G(x)\leq \mathrm{Ker}(f).$ Let $a\in k,$ $a\neq 0.$ We define $$f_a: (C_G(x),\cdot)\rightarrow (k,+),\quad f_a(g)=f(g)\cdot a,$$ for any $g\in G.$ Since $f_a(g_0)\neq 0,$ clearly $f_a\neq 0,$ and $\mathrm{Ker}(f)\geq \Gamma_xC_G(x).$
This means that $\mathrm{H}_{\Gamma_x}^1(C_G(x),k)\neq 0.$
Finally, using Proposition \ref{prop33} we apply the injective $k$-linear map $\nu_{\Gamma_x,G}^1$  to obtain $\HH_{\Gamma_x}^1(kG)\neq 0.$
\end{proof}

\section{\texorpdfstring{The $\Gamma$-equivariant Hochschild cohomology using a $\Gamma$-cochain complex and proof of Theorem \ref{thm12}}{The Gamma-equivariant Hochschild cohomology using a Gamma-cochain complex and proof of Theorem 1.2}} \label{sec6}

Let $A$ be a $\Gamma$-algebra and $L^*$ be a cochain complex of left $A$-modules.

\begin{definition}\label{defn61}\quad
\begin{itemize}
\item[(i)] We say that $\Gamma$ is acting on $$L^*:
\xymatrix{
& \ldots \ar[r]
& L^{n-1} \ar[r]^{\delta^{n-1}}
& L^n \ar[r]^{\delta^{n}}
& L^{n+1} \ar[r]
& \ldots
}$$
if:
\begin{itemize}
\item $\Gamma$ acts on each $L^n$, $n\in \Z,$ such that $L^n$ is becoming a $\Gamma$-equivariant $A$-module;
\item $\delta^n((L^{n})^{\Gamma})\subseteq (L^{{n+1}})^{\Gamma},$ for every $n\in \Z;$ in particular, this condition is satisfied if $\delta^*$ is compatible with the action of $\Gamma.$

\end{itemize}

\item[(ii)] We assume that $\Gamma$ is acting on $L^*$ and let $n\in \Z.$ The $\Gamma$-\textit{equivariant cohomology groups} of $L^*,$  denoted by $\Hrm_{\Gamma}^n(L^*),$
are defined as the cohomology groups of the next cochain subcomplex:
$$L_{\Gamma}^*: \xymatrix{
& \ldots \ar[r]
& (L^{n-1})^{\Gamma} \ar[r]^{\delta_0^{n-1}}
& (L^n)^{\Gamma} \ar[r]^{\delta_0^{n}}
& (L^{n+1}))^{\Gamma} \ar[r]
& \ldots
}$$
where $\delta_0^n$ is the restriction of $\delta^n$ to $(L^n)^{\Gamma}.$
\end{itemize}
\end{definition}

Let $M$ be a $\Gamma$-equivariant $(A-A)$-bimodule. In Subsection \ref{subsec22} we revisited the  Hochschild cochain complex $C^*(A,M).$
\begin{definition}\label{defn62}
Let $\Gamma$ be a finite group acting on $C^*(A,M).$ Then $\Hrm_{\Gamma}^*(C^*(A,M))$ is called the $\Gamma$-equivariant Hochschild cohomology of the $k$-algebra $A,$ with coefficients in the $\Gamma$-equivariant $(A-A)$-bimodule $M.$
\end{definition}

We will work in this paper in the case when the action of $\Gamma$ on $C^*(A,M)$ is induced by the actions of $\Gamma$ on $A,$ respectively on $M.$
That is, for $n\in \Z,$ $n>0$ we have
     $$\Gamma\times C^n(A,M)\rightarrow C^n(A,M)$$
     $$(\sigma,f)\mapsto {}^{\sigma}f;\quad {}^{\sigma}f: A^{\otimes n}\rightarrow M$$
     \begin{align*}
     {}^{\sigma}f(a_1\otimes\dots\otimes a_n)
     &= {}^{\sigma} f({}^{\sigma^{-1}}a_1\otimes\dots\otimes {}^{\sigma^{-1}}a_n),
     \end{align*}
     for any $a_1,\dots,a_n\in A, \sigma\in \Gamma.$
In this case we recover the $n$-th $\Gamma$-equivariant Hochschild cohomology of $A$
with coefficients in $M,$ as defined in Subsection \ref{subsec22},
$$
\HH_{\Gamma}^n(A,M)=\Hrm_{\Gamma}^n(C^*(A,M)).
$$

\begin{remark}
When $G$ is a finite group on which $\Gamma$ is acting, the Bar $kG$-resolution of $k,$ denoted by $\mathrm{Bar}_*(k),$ is a free $kG$-resolution of $k$ on which $\Gamma$ is acting by
  $$^{\sigma}(g[g_1,\dots,g_n])=^{\sigma}g[^{\sigma}g_1,\dots,^{\sigma}g_n],\mbox{ }n\geq 1,$$
  $$g,g_1,\dots,g_n\in G.$$ We assume that $\Gamma$ is acting trivially on $k$ and for any $n\geq 1$ we have that
    $$\mathrm{Hom}_{k(G\rtimes \Gamma)}(\mathrm{Bar}_n(k),k)\cong(C^n(G,k))^{\Gamma},$$
 see Subsection \ref{subsec21}. Thus $\Hrm_{\Gamma}^n(G,k)\cong\Hrm_{\Gamma}^n(C^*(G,k)).$ This remark is  inspired by Innassaridze,  which in  \cite[Section 5]{In}, studies the $\Gamma$-equivariant homology.
\end{remark}
The category of $\Gamma$-equivariant $(A-A)$-bimodules  is  denoted $(A,\Gamma)-\Bim$ with
$$
\Ob\big((A,\Gamma)-\Bim\big)
= \{\, X \mid X \text{ is a } \Gamma-\text{equivariant } (A-A)-\text{bimodule} \,\},
$$

$$
\Hom_{(A,\Gamma)-\Bim}(X,Y)
= \{\, f \in \Hom_{A^e}(X,Y) \mid
f({}^{\sigma}x) = {}^{\sigma}f(x),\ \forall \sigma \in \Gamma,\ \forall x \in X \,\}.
$$
This category is introduced in \citep{KoPi}, where it is studied in the context of oriented $\Gamma$-algebras. For any $H$-algebra $B$ (where $H$ is a finite group), the skew group algebra $B\star H$ is  $B\otimes kH$ as vector space, with the multiplication given by
\begin{equation}\label{eq61} (b_1\star g_1)(b_2 \star g_2)=b_1 ({ }^{g_1}b_2)\star g_1g_2
\end{equation}
 for any  $b_1,b_2\in B, g_1,g_2\in H.$
\begin{lemma}\label{lem71}
With the above notations, let $M \in \Ob((A,\Gamma)-\Bim).$
\begin{enumerate}
\item[(i)] There is an isomorphism of $\Gamma$-algebras $(A, \Gamma)^e \cong A^e \star\Gamma;$
\item[(ii)] The following isomorphisms of categories hold:
$$
(A,\Gamma)-\Bim \cong (A, \Gamma)^e-\Mod \cong A^e \star \Gamma-\Mod;
$$
\item[(iii)] There is an isomorphism of cochain complexes:
$$
\Hom_{(A,\Gamma)^e}(\mathrm{Bar}_*(A), M) \cong (C^*(A, M))_{\Gamma},$$
where $\Barc_*(A) \rightarrow A$ is the Hochschild Bar resolution.
\end{enumerate}
\end{lemma}

\begin{proof}\quad
\begin{itemize}
\item[(i)] The isomorphism is given by $$f:(A,\Gamma)^e\rightarrow A^e\star\Gamma,\quad f(a\otimes \sigma\otimes b)=(a\otimes b)\star\sigma,$$
with the inverse $$f^{-1}: A^e\star\Gamma\rightarrow (A,\Gamma)^e, \quad f^{-1}((a\otimes b)\star\sigma)=a\otimes \sigma\otimes b,$$
for any $a,b\in A, \sigma\in \Gamma.$
Both $f$ and $f^{-1}$ are $\Gamma$-homomorphisms.
\item[(ii)] The first isomorphism of categories is given in \cite[Lemma 6.1]{KoPi} and the second can be proved similarly.
\item [(iii)] Explicitly, the Hochschild Bar resolution is:
$$\Barc_*(A) \rightarrow A, \quad \Barc_n(A) = A^{\otimes(n+2)} \cong A^e \otimes A^{\otimes n},$$
\begin{gather*}
\dots\to A^{\otimes (n+3)} \xrightarrow{d_{n+1}} A^{\otimes (n+2)} \xrightarrow{d_{n}} A^{\otimes (n+1)} \to \dots \\
\dots \to A^{\otimes 3} \xrightarrow{d_{1}} A\otimes A \xrightarrow{\mu} A \to 0
\end{gather*}

$$d_n(a_0 \otimes \dots \otimes a_{n+1}) = \sum_{i=0}^n (-1)^i a_0 \otimes a_1 \otimes \dots \otimes a_i a_{i+1} \otimes \dots \otimes a_{n+1}.$$
$\Barc_*(A)$ is a projective resolution of $A$ as $(A-A)$-bimodule, and in fact each $\Barc_n(A)$, $n\geq 0,$ is a relative  $k\Gamma$-projective (free) left $A^e$-module, see the proof of \cite[Theorem 5.6 (3)]{In} and Subsection \ref{subsec23}. It follows that
$$\Hom_{(A,\Gamma)^e}(\Barc_*(A),M)\cong \Hom_{(A,\Gamma)-\Bim}(\Barc_*(A),M)=$$ $$\Hom_{A^e}(\Barc_*(A),M)\cap\Hom_{k\Gamma}(\Barc_*(A),M)\cong
(C^*(A,M))_{\Gamma}, $$
where the first isomorphism is true by Lemma \ref{lem71} (ii), and the second equality is exactly the definition of $\Hom$ in $(A,\Gamma)-\Bim.$  The last isomorphism is obtained using the isomorphism \cite[(1.1.11)]{Wi}, which  is clearly compatible with action of $\Gamma$ and the differentials.
\end{itemize}
\end{proof}

\begin{proof}\textbf{(of Theorem \ref{thm12}.)}
\begin{itemize}
\item[i)] The Hochschild cochain complex $C^*(A,M)$ is $\Gamma$-acted (since $M$ is a $\Gamma$-equivariant $(A-A)$-bimodule) and, by Definition \ref{defn62} we know that
\begin{equation} \label{eq71}
\HH_{\Gamma}^n(A,M)=\Hrm^n((C^*(A,M))_{\Gamma}).
\end{equation}
Again, by the proof of \cite[Theorem 7.6 (3)]{In} we know that $\Barc_*(A)\rightarrow A$ is a relatively $k\Gamma$-projective resolution of $A$ as a left $A^e$-module, hence using Subsection \ref{subsec23} we obtain
\begin{equation}\label{eq72}
\Ext_{(A,\Gamma)^e,k\Gamma}^n (A,M)=\Hrm^n(\Hom_{(A,\Gamma)^e}(\Barc_*(A),M))
\end{equation}
Finally, using Lemma \ref{lem71} (iii) and the relations (\ref{eq71}), (\ref{eq72}) we obtain the desired conclusion.
\item[ii)] By \cite[Theorem 6.1]{KoPi} there is a natural $k$-linear isomorphism $\Hrm_{\Gamma}^n(A,M)\cong \Ext_{(A,\Gamma)^e}^n (A,M).$ Now the proof is an easy consequence of statement i), (\ref{eqSha1}) and (\ref{eqSha2}).
\end{itemize}
\end{proof}

\backmatter
\section*{Declarations}
This work was supported by a grant of the Ministry of Research, Innovation and Digitization, CNCS -UEFISCDI, project number PN-IV-P1-PCE-2023-0060, within PNCDI IV.\\

\section*{Acknowledgements}
We thank Matthew Antrobus  for pointing us an error in the proof of Proposition \ref{prop32}, of a previous version of this paper.

\bibliography{sn-bibliography_pojar_todea}

@article{Ar,
  author  = "Aryutunov, A. A. and Mischenko, A. S.",
  title   = "Smooth version for {J}ohnson's problem on derivations of group algebras",
  journal = "Mat. Sb.",
  volume  = "210 (6)",
  pages   = "3--29",
  year    = "2019",
  doi     = " https://doi.org/10.1070/SM9119"
}

@article{Br,
  author  = "Briggs, B. and {Rubio y Degrassi}, L.",
  title   = "Outer derivations on blocks of group algebras",
  journal = "arXiv:2601.09602 [math.RT]",
  volume  = "",
  pages   = "",
  year    = "2026",
  doi     = "https://doi.org/10.48550/arXiv.2601.09602"
}

@article{Bu,
  author  = "Burghelea, D.",
  title   = "The cyclic homology of the group rings",
  journal = "Commentarii Mathematici Helvetici",
  volume  = "60",
  pages   = "354--365",
  year    = "1985",
  doi     = "https://doi.org/10.1007/BF02567420"
}

@article{Ceg,
  author  = "Cegarra, A. M. and Garc{\'i}a-Calcinez, J. M. and Ortega, J. A.",
  title   = "Cohomology of groups with operators",
  journal = "Homology, Homotopy and Applications",
  volume  = "4",
  number  = "1",
  pages   = "1--23",
  year    = "2002",
  doi     = "https://doi.org/10.4310/HHA.2002.v4.n1.a1"
}

@article{CoTo,
  author  = "Cocone{\c{t}}, T. and Todea, C.-C.",
  title   = "Symmetric {H}ochschild cohomology of twisted group algebras",
  journal = "Homology, Homotopy and Applications",
  volume  = "24",
  number  = "1",
  pages   = "93--115",
  year    = "2022",
  doi     = "https://doi.org/10.4310/HHA.2022.v24.n1.a5"
}

@article{FJL,
  author  = "Fleischmann, P. and Janiszczak, I. and Lempken, W.",
  title   = "Finite groups have local {N}on-{S}chur centralizers",
  journal = "Manuscripta Mathematica",
  volume  = "80",
  pages   = "213--224",
  year    = "1993",
  doi     = "https://doi.org/10.1007/BF03026547"
}

@article{GeSc,
  author  = "Gerstenhaber, M. and Schack, S. D.",
  title   = "Relative {H}ochschild cohomology, rigid algebras and the {B}ockstein",
  journal = "Journal of Pure and Applied Algebra",
  volume  = "43",
  pages   = "53--74",
  year    = "1986",
  doi     = "https://doi.org/10.1016/0022-4049(86)90004-6"
}

@article{Ho,
  author  = "Hochschild, G.",
  title   = "Relative homological algebra",
  journal = "Transactions of the American Mathematical Society",
  volume  = "82",
  number  = "1",
  pages   = "246--269",
  year    = "1956",
  doi     = "https://doi.org/10.1090/S0002-9947-1956-0080654-0"
}

@article{In05,
  author  = "Inassaridze, H.",
  title   = "Equivariant homology and cohomology of groups",
  journal = "Topology and its Applications",
  volume  = "153",
  pages   = "66--89",
  year    = "2005",
  doi     = "https://doi.org/10.1016/j.topol.2004.12.005"
}

@article{In,
  author  = "Inassaridze, H.",
  title   = "({C}o)homology of {$\Gamma$}-groups and {$\Gamma$}-homological algebra",
  journal = "European Journal of Mathematics",
  volume  = "8",
  number  = "Suppl 2",
  pages   = "720--763",
  year    = "2022",
  doi     = "https://doi.org/10.1007/s40879-022-00558-0"
}

@book{Ja,
  author    = "Jacobson, N.",
  title     = "Lie algebras",
  publisher = "Volume No.10 of {I}nterscience {T}racts in {P}ure and {A}pplied {M}athematics. {I}nterscience {P}ublisers",
  address   = "New York-London",
  year      = "1962",
  doi       = ""
}

@article{Je,
  author  = "Jensen, K. K.",
  title   = "Foundations of an {E}quivariant {C}ohomology {T}heory of {B}anach {A}lgebras, {I}",
  journal = "Adv. in Math.",
  volume  = "117",
  pages   = "52--146",
  year    = "1996",
  doi     = "https://doi.org/10.1006/aima.1996.0003"
}

@article{KoPi,
  author  = "Koam, A. N. A. and Pirashvili, T.",
  title   = "Cohomology of oriented algebras",
  journal = "Communications in Algebra",
  volume  = "46",
  number  = "7",
  pages   = "2947--2963",
  year    = "2018",
  doi     = "https://doi.org/10.1080/00927872.2017.1404089"
}

@article{LiZh,
  author  = "Liu, Y. and Zhou, G.",
  title   = "The {B}atalin--{V}ilkovisky structure over the {H}ochschild cohomology ring of a group algebra",
  journal = "Journal of Noncommutative Geometry",
  volume  = "10",
  pages   = "811--858",
  year    = "2016",
  doi     = "https://doi.org/10.4171/JNCG/249"
}

@article{Li22,
  author  = "Linckelmann, M.",
  title   = "Finite dimensional algebras arising as blocks of finite group algebras",
  journal = "Contemporary Mathematics",
  volume  = "705",
  pages   = "",
  year    = "2018",
  doi     = "http://dx.doi.org/10.1090/conm/705/14202"
}

@article{Mu,
  author  = "Mueller, L. and Wike, L.",
  title   = "Equivariant higher {H}ochschild homology and topological field theories",
  journal = "Homology, Homotopy and Applications",
  volume  = "22",
  number  = "1",
  pages   = "27--54",
  year    = "2020",
  doi     = "http://dx.doi.org/10.4310/HHA.2020.v22.n1.a3"
}

@article{Gou,
  author  = "Mukherjee, G. and Yadav, R. B.",
  title   = "Equivariant one-parameter deformations of associative algebras",
  journal = "Journal of Algebra and Its Applications",
  volume  = "19",
  number  = "6",
  pages   = "2050114",
  year    = "2020",
  doi     = "https://doi.org/10.1142/S0219498820501145",
  note    = "18 pages"
}

@book{Wi,
  author    = "Witherspoon, S. J.",
  title     = "Hochschild cohomology for algebras",
  publisher = "American Mathematical Society",
  address   = "Providence, Rhode Island",
  year      = "2019",
  doi       = "https://doi.org/10.1090/gsm/204"
}
\end{document}